\documentclass{amsart}
\usepackage{amsmath,amssymb}
\usepackage[dvips]{graphics}
\usepackage[dvipdfmx]{graphicx}
\usepackage[all]{xy}

\newtheorem{Theorem}{Theorem}[section]
\newtheorem{Lemma}[Theorem]{Lemma}
\newtheorem{Proposition}[Theorem]{Proposition}
\newtheorem{Corollary}[Theorem]{Corollary}

\theoremstyle{definition}

\newtheorem{Example}[Theorem]{Example}
\newtheorem{Problem}[Theorem]{Problem}

\theoremstyle{remark}
\newtheorem{Remark}[Theorem]{Remark}

\def\labelenumi{\rm ({\makebox[0.8em][c]{\theenumi}})}
\def\theenumi{\arabic{enumi}}
\def\labelenumii{\rm {\makebox[0.8em][c]{(\theenumii)}}}
\def\theenumii{\roman{enumii}}

\renewcommand{\thefigure}
{\arabic{section}.\arabic{figure}}
\makeatletter
\@addtoreset{figure}{section}
\def\@thmcountersep{-}
\makeatother

\numberwithin{equation}{section}

\begin{document} 

\title[Generalizations of the Conway-Gordon theorems]{Generalizations of the Conway-Gordon theorems and intrinsic knotting on complete graphs}

\author{Hiroko Morishita}
\address{Division of Mathematics, Graduate School of Science, Tokyo Woman's Christian University, 2-6-1 Zempukuji, Suginami-ku, Tokyo 167-8585, Japan}
\email{d17m104@cis.twcu.ac.jp}

\author{Ryo Nikkuni}
\address{Department of Mathematics, School of Arts and Sciences, Tokyo Woman's Christian University, 2-6-1 Zempukuji, Suginami-ku, Tokyo 167-8585, Japan}
\email{nick@lab.twcu.ac.jp}
\thanks{The second author was supported by JSPS KAKENHI Grant Number JP15K04881.}


\date{}



\begin{abstract}
In 1983, Conway and Gordon proved that for every spatial complete graph on six vertices, the sum of the linking numbers over all of the constituent two-component links is odd, and that for every spatial complete graph on seven vertices, the sum of the Arf invariants over all of the Hamiltonian knots is odd. In 2009, the second author gave integral lifts of the Conway-Gordon theorems in terms of the square of the linking number and the second coefficient of the Conway polynomial. In this paper, we generalize the integral Conway-Gordon theorems to complete graphs with arbitrary number of vertices greater than or equal to six. As an application, we show that for every rectilinear spatial complete graph whose number of vertices is greater than or equal to six, the sum of the second coefficients of the Conway polynomials over all of the Hamiltonian knots is determined explicitly in terms of the number of triangle-triangle Hopf links. 

\end{abstract}

\maketitle

\section{Introduction} 

Throughout this paper we work in the piecewise linear category. Let $G$ be a finite simple graph. An embedding $f$ of $G$ into the $3$-dimensional Euclidean space ${\mathbb R}^{3}$ is called a {\it spatial embedding} of $G$, and the image $f(G)$ is called a {\it spatial graph} of $G$. Two spatial embeddings $f$ and $g$ of $G$ are said to be {\it equivalent} if there exists a self homeomorphism $\Phi$ on ${\mathbb R}^{3}$ such that $\Phi\left(f(G)\right) = g(G)$. We call a subgraph $\gamma$ of $G$ homeomorphic to the circle a {\it cycle} of $G$, and a cycle of $G$ containing exactly $k$ edges a {\it $k$-cycle} of $G$. In particular, a $k$-cycle is also called a {\it Hamiltonian cycle} if  $k$ equals the number of  vertices of $G$. We  denote the set of all $k$-cycles of $G$ by $\Gamma_{k}(G)$. Moreover, we denote the set of all pairs of two disjoint cycles of $G$ consisting of a $k$-cycle and an $l$-cycle by $\Gamma_{k,l}(G)$. For a cycle $\gamma$ (resp. a pair of disjoint cycles $\lambda$)  and a spatial embedding $f$ of $G$, $f(\gamma)$ (resp. $f(\lambda)$) is none other than a knot (resp. a $2$-component link) in $f(G)$. In particular for a Hamiltonian cycle $\gamma$ of $G$, we also call $f(\gamma)$ a {\it Hamiltonian knot} in $f(G)$. 

Let $K_{n}$ be the {\it complete graph} on $n$ vertices, that is the graph consisting of $n$ vertices such that each pair of its distinct vertices is connected by exactly one edge. Then the following fact is well-known as the {\it Conway-Gordon theorem}. 

\begin{Theorem}\label{CG1} {\rm (Conway-Gordon \cite{CG83})}
\begin{enumerate}
\item For any spatial embedding $f$ of $K_{6}$, we have 
\begin{eqnarray*}
\sum_{\lambda\in \Gamma_{3,3}\left(K_{6}\right)}{\rm lk}\left(f(\lambda)\right) \equiv 1 \pmod{2}, 
\end{eqnarray*}
where ${\rm lk}$ denotes the {\it linking number} in ${\mathbb R}^{3}$.
\item For any spatial embedding $f$ of $K_{7}$, we have 
\begin{eqnarray*}
\sum_{\gamma\in \Gamma_{7}\left(K_{7}\right)}a_{2}\left(f(\gamma)\right) \equiv 1 \pmod{2}, 
\end{eqnarray*}
where $a_{2}$ denotes the second coefficient of the {\it Conway polynomial}. 
\end{enumerate}
\end{Theorem}

The second coefficient of the Conway polynomial of a knot is also congruent with the {\it Arf invariant} of the knot modulo two \cite[Corollary 10.8]{Kauffman83}. Theorem \ref{CG1} implies that $K_{6}$ is {\it intrinsically linked}, that is, every spatial graph of $K_{6}$ contains a nonsplittable $2$-component link, and $K_{7}$ is {\it intrinsically knotted}, that is, every spatial graph of $K_{7}$ contains a nontrivial knot. The Conway-Gordon theorem made a beginning of the study of intrinsic linking and knotting of graphs and has motivated a lot of studies of intrinsic properties of graphs (see for example \cite[\S\S 2-6]{FMMRN17}). On the other hand, as far as the authors know, there have been little results about a generalization of the Conway-Gordon type congruences for complete graphs on eight or more vertices. Our purposes in this paper are to generalize the Conway-Gordon theorems for complete graphs with arbitrary number of vertices greater than or equal to six and to investigate the behavior of the nontrivial Hamiltonian knots in a spatial complete graph. First of all, we recall an integral Conway-Gordon theorem for $K_{6}$ which was proven by the second author as follows.

\begin{Theorem}\label{main1} {\rm (Nikkuni \cite{Nikkuni09})} 
For any spatial embedding $f$ of $K_{6}$, we have 
\begin{eqnarray}\label{k6ref}
\ \ \ \ \ \ 2\sum_{\gamma\in \Gamma_{6}\left(K_{6}\right)}a_{2}\left(f(\gamma)\right)
- 2\sum_{\gamma\in \Gamma_{5}\left(K_{6}\right)}a_{2}\left(f(\gamma)\right)
= 
\sum_{\lambda\in \Gamma_{3,3}\left(K_{6}\right)}{{\rm lk}\left(f(\lambda)\right)}^{2}-1. 
\end{eqnarray}
\end{Theorem}

Note that Theorem \ref{CG1} (1) can be recovered by taking the modulo two reduction of (\ref{k6ref}), namely Theorem \ref{main1} is an integral lift of Theorem \ref{CG1} (1). In \cite{Nikkuni09}, an integral lift of Theorem \ref{CG1} (2) was also given (see Theorem \ref{K331} (1) of the present paper). In this paper, we generalize Theorem \ref{main1} for complete graphs with arbitrary number of vertices greater than or equal to six as follows. 

\begin{Theorem}\label{mainthm} 
Let $n\ge 6$ be an integer. For any spatial embedding $f$ of $K_{n}$, we have 
\begin{eqnarray}\label{maintheorem}
&&\sum_{\gamma\in \Gamma_{n}\left(K_{n}\right)}a_{2}\left(f(\gamma)\right)
- (n-5)!\sum_{\gamma\in \Gamma_{5}\left(K_{n}\right)}a_{2}\left(f(\gamma)\right)\\
&=& 
\frac{(n-5)!}{2} 
\bigg(
\sum_{\lambda\in \Gamma_{3,3}\left(K_{n}\right)}{{\rm lk}\left(f(\lambda)\right)}^{2}
- \binom{n-1}{5}
\bigg).\nonumber 
\end{eqnarray}
\end{Theorem}

By Theorem \ref{mainthm}, we also obtain formulae of two types. First we have the following inequality, where 
the case of $n=7$ has already been observed in \cite[Lemma 4.2]{Nikkuni09}.

\begin{Corollary}\label{maincor} 
Let $n\ge 6$ be an integer. For any spatial embedding $f$ of $K_{n}$, we have 
\begin{eqnarray*}
\sum_{\gamma\in \Gamma_{n}\left(K_{n}\right)}a_{2}\left(f(\gamma)\right)
- (n-5)!\sum_{\gamma\in \Gamma_{5}\left(K_{n}\right)}a_{2}\left(f(\gamma)\right)
\ge   
\frac{(n-5)(n-6)(n-1)!}{2\cdot 6!}.
\end{eqnarray*}
\end{Corollary}

The lower bound of Corollary \ref{maincor} is sharp, see Remark \ref{mainrem0}. Next we also have the following congruence, that is a generalization of Theorem \ref{CG1} (2). 

\begin{Corollary}\label{maincor0} 
Let $n\ge 7$ be an integer. For any spatial embedding $f$ of $K_{n}$, we have the following congruence modulo $(n-5)!$: 
\begin{eqnarray*}
\sum_{\gamma\in \Gamma_{n}\left(K_{n}\right)}a_{2}\left(f(\gamma)\right) 
\equiv 
\left\{
   \begin{array}{@{\,}lll}
   {\displaystyle - \frac{(n-5)!}{2} \binom{n-1}{5}} & (n\equiv 0\pmod{8}) \\
   0 & (n\not\equiv 0,7\pmod{8}) \\
   {\displaystyle \frac{(n-5)!}{2}\binom{n}{6}} & (n\equiv 7\pmod{8}). 
   \end{array}
\right.
\end{eqnarray*}
\end{Corollary}

Corollary \ref{maincor0} contains the preceding results concerning Conway-Gordon type congruences on the sum of $a_{2}$, see Remark \ref{mainrem1}.

Theorem \ref{mainthm} (and Corollary \ref{maincor}) is also useful for investigating the behavior of the nontrivial Hamiltonian knots in rectilinear spatial complete graphs. Here, a spatial embedding $f_{\rm r}$ of a graph $G$ is said to be {\it rectilinear} if for any edge $e$ of $G$, $f_{\rm r}(e)$ is a straight line segment in ${\mathbb R}^{3}$. A rectilinear spatial graph appears in polymer chemistry as a mathematical model for chemical compounds (see \cite[\S 7]{Adams04}, for example), and the range of rectilinear spatial graph types is much narrower than the general spatial graphs. So we are interested in the behavior of the nontrivial Hamiltonian knots in a rectilinear spatial complete graph. Note that every knot (resp. link) contained in a rectilinear spatial graph of $K_{n}$ is a ``polygonal'' knot (resp. link) with less than or equal to $n$ sticks. It is well-known that every polygonal knot with less than or equal to five sticks is trivial (Proposition \ref{stick} (1)). Thus for rectilinear spatial complete graphs, by Theorem \ref{mainthm}  we have the following immediately.

\begin{Theorem}\label{mainthmrecti} 
Let $n\ge 6$ be an integer. For any rectilinear spatial embedding $f_{\rm r}$ of $K_{n}$, we have  
\begin{eqnarray*}
\sum_{\gamma\in \Gamma_{n}\left(K_{n}\right)}a_{2}\left(f_{\rm r}(\gamma)\right)
=  
\frac{(n-5)!}{2} 
\bigg(
\sum_{\lambda\in \Gamma_{3,3}\left(K_{n}\right)}{{\rm lk}\left(f_{\rm r}(\lambda)\right)}^{2}
- \binom{n-1}{5}
\bigg). 
\end{eqnarray*}
\end{Theorem}

Also note that a $2$-component link with exactly six sticks is either a trivial link or a Hopf link (Proposition \ref{stick} (2)). Thus for any rectilinear spatial embedding $f_{\rm r}$ of $K_{n}$, $\sum_{\lambda\in \Gamma_{3,3}\left(K_{n}\right)}{{\rm lk}\left(f_{\rm r}(\lambda)\right)}^{2}$ is equal to the number of ``triangle-triangle'' Hopf links in $f_{\rm r}(K_{n})$. Then, by using Corollary \ref{maincor} and Theorem \ref{mainthmrecti}, we can obtain the following upper and lower bounds of $\sum_{\gamma\in \Gamma_{n}\left(K_{n}\right)}a_{2}\left(f_{\rm r}(\gamma)\right)$.

\begin{Corollary}\label{maincor2} 
Let $n\ge 6$ be an integer. For any rectilinear spatial embedding $f_{\rm r}$ of $K_{n}$, we have   
\begin{eqnarray*}
\frac{(n-5)(n-6)(n-1)!}{2\cdot 6!}
\le 
\sum_{\gamma\in \Gamma_{n}\left(K_{n}\right)}a_{2}\left(f_{\rm r}(\gamma)\right)
\le   
\frac{3(n-2)(n-5)(n-1)!}{2\cdot 6!}. 
\end{eqnarray*}
\end{Corollary}

The lower bound in Corollary \ref{maincor2} is also sharp, see Remark \ref{recrec}. However, the authors expect that the upper bound is not sharp if $n\ge 7$, see Example \ref{ub}.

For every spatial embedding $f$ of $K_{n}$ (which does not need to be rectilinear), Hirano showed that there exist at least three nontrivial Hamiltonian knots with an odd value of $a_{2}$ in $f\left(K_{8}\right)$ \cite{Hirano10}, and Foisy showed that there exist at least $(n-1)(n-2)\cdots 9\cdot 8$ nontrivial Hamiltonian knots with an odd value of $a_{2}$ in $f\left(K_{n}\right)$ if $n\ge 9$ \cite{BBFHL07}. On the other hand, Corollary \ref{maincor2}  makes us possible to evaluate the number of nontrivial Hamiltonian knots with a positive value of $a_{2}$ in a rectilinear spatial graph of $K_{n}$ as follows.

\begin{Corollary}\label{maincor3} 
Let $n\ge 7$ be an integer. The minimum number of nontrivial Hamiltonian knots with a positive value of $a_{2}$ in every rectilinear spatial graph of $K_{n}$ is at least 
\begin{eqnarray*}
r_{n} = 
\left\lceil 
\frac{
(n-5)(n-6)(n-1)! / (2\cdot 6!)
}
{
\left\lfloor
(n-3)^{2}(n-4)^{2} / 32
\right\rfloor
}
\right\rceil, 
\end{eqnarray*}
where $\lceil\cdot\rceil$ and $\lfloor\cdot\rfloor$ denote  the ceiling function and the floor function, respectively. 
\end{Corollary}

We see that  $r_{n}$ is greater than Foisy's lower bound of the minimum number of nontrivial Hamiltonian knots with an odd value of $a_{2}$ if $n=9,10,11$, see Remark \ref{fh}.

The paper is organized as follows. We shall devote Section 2 to proofs of Theorem \ref{mainthm} and Corollaries \ref{maincor}, \ref{maincor0}, \ref{maincor2} and \ref{maincor3}. In Section $3$, we give  examples and present some open problems.

\section{Proofs of Theorem \ref{mainthm} and its Corollaries} 

We show some lemmas which are needed to prove Theorem \ref{mainthm}.

\begin{Lemma}\label{nlem}
\begin{enumerate}
\item Let $n\ge 6$ be an integer. For any spatial embedding $f$ of $K_{n}$, we have 
\begin{eqnarray*}
2\sum_{\gamma\in\Gamma_{6}\left(K_{n}\right)}a_{2}\left(f(\gamma)\right)  
- 2(n-5) \sum_{\gamma\in\Gamma_{5}\left(K_{n}\right)}a_{2}\left(f(\gamma)\right) 
= \sum_{\lambda\in\Gamma_{3,3}\left(K_{n}\right)}{\rm lk}\left(f(\lambda)\right)^{2}
- \binom{n}{6}. 
\end{eqnarray*}
\item Let $n\ge 7$ be an integer. For any spatial embedding $f$ of $K_{n}$, we have 
\begin{eqnarray*}
\sum_{\lambda\in\Gamma_{3,4}\left(K_{n}\right)}{\rm lk}\left(f(\lambda)\right)^{2}
= 2(n-6)\sum_{\lambda\in\Gamma_{3,3}\left(K_{n}\right)}{\rm lk}\left(f(\lambda)\right)^{2}. 
\end{eqnarray*}
\end{enumerate}
\end{Lemma}

\begin{proof}[Proof of Lemma \ref{nlem} (1)]
Note that each $5$-cycle of $K_{n}$ is shared by exactly $n-5$ subgraphs isomorphic to $K_{6}$ if $n\ge 6$. Then by applying Theorem \ref{main1} to the embedding $f$ restricted to each of the subgraphs of $K_{n}$ isomorphic to $K_{6}$ and taking the sum of both sides of (\ref{k6ref}) over all of them, we have the result. 
\end{proof}

In order to prove Lemma \ref{nlem} (2), we recall integral Conway-Gordon type theorems for spatial embeddings of $K_{7}$ and $K_{3,3,1}$ which were proven by the second author \cite{Nikkuni09} and O'Donnol \cite{Daniella15}, respectively. Here, the {\it complete $k$-partite graph} $K_{n_{1},n_{2},\ldots,n_{k}}$ is the graph whose vertex set can be decomposed into $k$ mutually disjoint nonempty sets $V_{1},V_{2},\ldots,V_{k}$ where the number of elements in $V_{i}$ equals $n_{i}$ such that no two vertices in $V_{i}$ are connected by an edge and every pair of vertices in distinct sets $V_{i}$ and $V_{j}$ is connected by exactly one edge. See Fig. \ref{K33K331} which illustrates $K_{3,3}$ and $K_{3,3,1}$. In particular for $K_{3,3,1}$, let us denote the subgraph of $K_{3,3,1}$ which is isomorphic to $K_{3,3}$ and does not contain the vertex $u$ by $H$. 

\begin{Theorem}\label{K331} 
\begin{enumerate}
\item {\rm (Nikkuni \cite{Nikkuni09})} For any spatial embedding $f$ of $K_{7}$, we have 
\begin{eqnarray*}
&&7\sum_{\gamma\in \Gamma_{7}\left(K_{7}\right)}a_{2}\left(f(\gamma)\right)
- 6\sum_{\gamma\in \Gamma_{6}\left(K_{7}\right)}a_{2}\left(f(\gamma)\right)
- 2\sum_{\gamma\in \Gamma_{5}\left(K_{7}\right)}a_{2}\left(f(\gamma)\right)\\
&=& 2\sum_{\lambda\in \Gamma_{3,4}\left(K_{7}\right)}{{\rm lk}\left(f(\lambda)\right)}^{2}
- 21.
\end{eqnarray*}
\item {\rm (O'Donnol \cite{Daniella15})} For any spatial embedding $f$ of $K_{3,3,1}$, we have 
\begin{eqnarray*}
&& 2\sum_{\gamma\in \Gamma_{7}\left(K_{3,3,1}\right)}a_{2}\left(f(\gamma)\right)
-4\sum_{\gamma\in \Gamma_{6}\left(H\right)}a_{2}\left(f(\gamma)\right)
-2\sum_{\gamma\in \Gamma_{5}\left(K_{3,3,1}\right)}a_{2}\left(f(\gamma)\right)\\
&=& \sum_{\lambda\in \Gamma_{3,4}\left(K_{3,3,1}\right)}{{\rm lk}\left(f(\lambda)\right)}^{2}-1.
\end{eqnarray*}
\end{enumerate}
\end{Theorem}

\begin{figure}[htbp]
      \begin{center}
\scalebox{0.575}{\includegraphics*{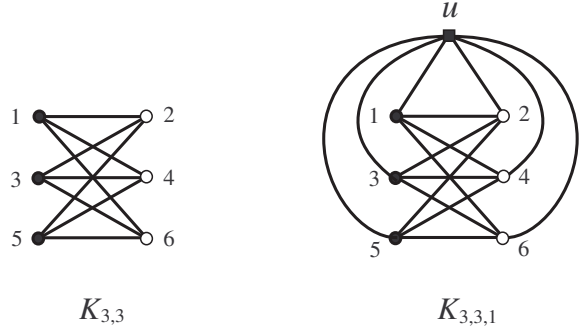}}
      \end{center}
   \caption{$K_{3,3}$ and $K_{3,3,1}$}
  \label{K33K331}
\end{figure} 

Then by applying Theorem \ref{K331} (2) to each of the subgraphs of $K_{7}$ isomorphic to $K_{3,3,1}$ and combining with Theorem \ref{K331} (1), we also have the following equation for every spatial embedding of $K_{7}$.

\begin{Theorem}\label{lk34}
For any spatial embedding $f$ of $K_{7}$, we have 
\begin{eqnarray}\label{34ref}
\sum_{\lambda\in \Gamma_{3,4}\left(K_{7}\right)}{{\rm lk}\left(f(\lambda)\right)}^{2} = 2\sum_{\lambda\in \Gamma_{3,3}\left(K_{7}\right)}{{\rm lk}\left(f(\lambda)\right)}^{2}.
\end{eqnarray}
\end{Theorem}

\begin{proof}[Proof of Theorem \ref{lk34}]
For vertices $1,2,\ldots,6$ and $u$ of $K_{3,3,1}$, we call the vertices $1,3,5$ the {\it black vertices}, the vertices $2,4,6$ the {\it white vertices} and the vertex $u$ the {\it square vertex}. Note that a $k$-cycle of $K_{3,3,1}$ contains the square vertex if $k$ is odd. There are exactly seventy subgraphs $G_{i}\ (i=1,2,\ldots,70)$ of $K_{7}$ isomorphic to $K_{3,3,1}$, because there are seven ways to choose the square vertex and $\frac{1}{2}\binom{6}{3}$ ways to choose the remaining black and white vertices. Then for a spatial embedding $f$ of $K_{7}$, by applying Theorem \ref{K331} (2) to the embedding $f$ restricted to $G_{i}$, we have 
\begin{eqnarray}\label{gk331}
&& 2\sum_{\gamma\in \Gamma_{7}\left(G_{i}\right)}a_{2}\left(f(\gamma)\right)
-4\sum_{\gamma\in \Gamma_{6}(H_{i})}a_{2}\left(f(\gamma)\right)
-2\sum_{\gamma\in \Gamma_{5}\left(G_{i}\right)}a_{2}\left(f(\gamma)\right)\\
&=& \sum_{\lambda\in \Gamma_{3,4}\left(G_{i}\right)}{{\rm lk}\left(f(\lambda)\right)}^{2}-1, \nonumber
\end{eqnarray}
where $H_{i}$ is the subgraph of $G_{i}$ isomorphic to $K_{3,3}$ not containing  the square vertex $(i=1,2,\ldots,70)$. Let us take the sum of both sides of (\ref{gk331}) for all $i$. Since each $7$-cycle $\gamma$ of $K_{7}$ is shared by exactly seven $G_{i}$'s (there are seven ways to choose the square vertex from the vertices of $\gamma$ and then the assignment of the black and white vertices is uniquely determined),  we have 
\begin{eqnarray}\label{331a}
\sum_{i=1}^{70}\sum_{\gamma\in \Gamma_{7}\left(G_{i}\right)}a_{2}\left(f(\gamma)\right) = 7\sum_{\gamma\in \Gamma_{7}\left(K_{7}\right)}a_{2}\left(f(\gamma)\right). 
\end{eqnarray}
Since for each  $6$-cycle $\gamma$ of $K_{7}$ there exists the unique $G_{i}$ such that $H_{i}$ contains $\gamma$ (the assignment of the black and white vertices is uniquely determined), we have   
\begin{eqnarray}\label{331b}
\sum_{i=1}^{70}\sum_{\gamma\in \Gamma_{6}(H_{i})}a_{2}\left(f(\gamma)\right) = \sum_{\gamma\in \Gamma_{6}\left(K_{7}\right)}a_{2}\left(f(\gamma)\right). 
\end{eqnarray}
Since each $5$-cycle $\gamma$ of $K_{7}$ is shared by exactly ten $G_{i}$'s (there are five ways to choose the square vertex from the vertices of $\gamma$ and two ways to choose the remaining black and white vertices),  we have 
\begin{eqnarray}\label{331c}
\sum_{i=1}^{70}\sum_{\gamma\in \Gamma_{5}\left(G_{i}\right)}a_{2}\left(f(\gamma)\right) = 10\sum_{\gamma\in \Gamma_{5}\left(K_{7}\right)}a_{2}\left(f(\gamma)\right). 
\end{eqnarray}
Since each pair of two disjoint cycles $\lambda$ in $\Gamma_{3,4}\left(K_{7}\right)$ is shared by exactly six $G_{i}$'s (there are three ways to choose the square vertex from the 3-cycle in $\lambda$ and two ways to choose the remaining black and white vertices),  we have 
\begin{eqnarray}\label{331d}
\sum_{i=1}^{70}\sum_{\lambda\in \Gamma_{3,4}\left(G_{i}\right)}{\rm lk}\left(f(\lambda)\right)^{2} = 6\sum_{\lambda\in \Gamma_{3,4}\left(K_{7}\right)}{\rm lk}\left(f(\lambda)\right)^{2}. 
\end{eqnarray}
Thus by combining (\ref{331a}), (\ref{331b}), (\ref{331c}) and (\ref{331d}) with (\ref{gk331}), we have 
\begin{eqnarray}\label{331e}
&&7\sum_{\gamma\in \Gamma_{7}\left(K_{7}\right)}a_{2}\left(f(\gamma)\right)
- 2\sum_{\gamma\in \Gamma_{6}\left(K_{7}\right)}a_{2}\left(f(\gamma)\right)
- 10\sum_{\gamma\in \Gamma_{5}\left(K_{7}\right)}a_{2}\left(f(\gamma)\right)\\
&=& 3\sum_{\lambda\in \Gamma_{3,4}\left(K_{7}\right)}{{\rm lk}\left(f(\lambda)\right)}^{2}
- 35.\nonumber
\end{eqnarray}
Then by (\ref{331e}) and Theorem \ref{K331} (1), we have   
\begin{eqnarray}\label{331f}
\ \ \ \ \ \ 4\sum_{\gamma\in\Gamma_{6}\left(K_{7}\right)}a_{2}\left(f(\gamma)\right)  
- 8 \sum_{\gamma\in\Gamma_{5}\left(K_{7}\right)}a_{2}\left(f(\gamma)\right) 
= \sum_{\lambda\in\Gamma_{3,4}\left(K_{7}\right)}{\rm lk}\left(f(\lambda)\right)^{2}
- 14. 
\end{eqnarray}
On the other hand, by Lemma \ref{nlem} (1) we have 
\begin{eqnarray}\label{331g}
\ \ \ \ \ 2\sum_{\gamma\in\Gamma_{6}\left(K_{7}\right)}a_{2}\left(f(\gamma)\right)  
- 4 \sum_{\gamma\in\Gamma_{5}\left(K_{7}\right)}a_{2}\left(f(\gamma)\right) 
= \sum_{\lambda\in\Gamma_{3,3}\left(K_{7}\right)}{\rm lk}\left(f(\lambda)\right)^{2} - 7. 
\end{eqnarray}
By (\ref{331f}) and (\ref{331g}), we have the desired conclusion. 
\end{proof}

\begin{proof}[Proof of Lemma \ref{nlem} (2)]
Note that each pair of two disjoint $3$-cycles of $K_{n}$ is shared by exactly $n-6$ subgraphs isomorphic to $K_{7}$ if $n\ge 7$. Then by applying Theorem \ref{lk34} to the embedding $f$ restricted to each of the subgraphs of $K_{n}$ isomorphic to $K_{7}$ and taking the sum of both sides of (\ref{34ref}) over all of them, we have the result. 
\end{proof}

Now we show a lemma which plays a major role in the proof of  Theorem \ref{mainthm}. The proof  is in the same spirit as that of Theorem \ref{K331} (1) in \cite{Nikkuni09}.

\begin{Lemma}\label{ind}
Let $n\ge 7$ be an integer. Assume that there exist three constants  $b,c$ and $d$ such that 
\begin{eqnarray*}
\sum_{\gamma\in \Gamma_{n-1}\left(K_{n-1}\right)}a_{2}\left(g(\gamma)\right)
+ b\sum_{\gamma\in \Gamma_{5}\left(K_{n-1}\right)}a_{2}\left(g(\gamma)\right) 
=  
c\sum_{\lambda\in \Gamma_{3,3}\left(K_{n-1}\right)}{{\rm lk}\left(g(\lambda)\right)}^{2} + d 
\end{eqnarray*}
for any spatial embedding $g$ of $K_{n-1}$. Then we have 
\begin{eqnarray*}
&& \sum_{\gamma\in \Gamma_{n}\left(K_{n}\right)}a_{2}\left(f(\gamma)\right)
+ b(n-5)\sum_{\gamma\in \Gamma_{5}\left(K_{n}\right)}a_{2}\left(f(\gamma)\right)\\ 
&=&   
\frac{c(n - 6)(n + 1) - 3b}{n}\sum_{\lambda\in \Gamma_{3,3}\left(K_{n}\right)}{{\rm lk}\left(f(\lambda)\right)}^{2} + d(n-1) + \frac{b}{2}\binom{n - 1}{5}
\end{eqnarray*}
for any spatial embedding $f$ of $K_{n}$. 
\end{Lemma}

\begin{proof}[Proof]
In the following, we denote the edge of $K_{n}$ connecting two distinct vertices $i$ and $j$ by $\overline{ij}$, and denote a path of length $2$ of $K_{n}$ consisting of two edges $\overline{ij}$ and $\overline{jk}$ by $\overline{ijk}$. We denote the subgraph of $K_{n}$ obtained from $K_{n}$ by deleting the vertex $m$ and all of the edges incident to $m$ by $K_{n-1}^{(m)}\ (m = 1,2,\ldots,n)$. Actually $K_{n-1}^{(m)}$ is isomorphic to $K_{n-1}$ for any $m$. For $1\le i<j\le n$ and $i,j\neq m$, let $F_{ij}^{(m)}$ be the subgraph of $K_{n}$ obtained from $K_{n}$ by deleting the edges $\overline{ij}$ and $\overline{mk}$ for all $k$ with $1\le k\le n,\ k\neq i,j$. Note that $F_{ij}^{(m)}$ is homeomorphic to $K_{n-1}$, namely $F_{ij}^{(m)}$ is obtained from $K_{n-1}^{(m)}$ by subdividing the edge $\overline{ij}$ by the vertex $m$, see Fig. \ref{Knsubdivide}.

\begin{figure}[htbp]
      \begin{center}
\scalebox{0.5}{\includegraphics*{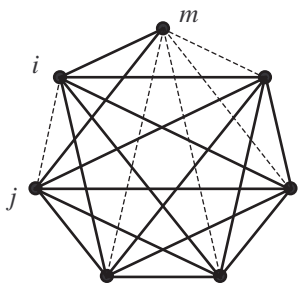}}
      \end{center}
   \caption{$F_{ij}^{(m)}$ ($n = 7$)}
  \label{Knsubdivide}
\end{figure} 

Let $f$ be a spatial embedding of $K_{n}$. Then for the embedding $f$ restricted to $F_{ij}^{(m)}$, by the assumption we have 
\begin{eqnarray}\label{eq1}
&&\sum_{\gamma\in\Gamma_{n}(F_{ij}^{(m)})}a_{2}\left(f(\gamma)\right)
+\sum_{\substack{\gamma\in\Gamma_{n-1}(K_{n-1}^{(m)}) \\ \overline{ij}\not\subset \gamma}} a_{2}\left(f(\gamma)\right)\\
&& +b \Biggl(
\sum_{\substack{\gamma\in\Gamma_{6}(F_{ij}^{(m)}) \\ \overline{imj}\subset \gamma}} a_{2}\left(f(\gamma)\right)
+\sum_{\substack{\gamma\in\Gamma_{5}(K_{n-1}^{(m)}) \\ \overline{ij}\not\subset \gamma}} a_{2}\left(f(\gamma)\right)
\Biggr)\nonumber\\
&=&
c\Biggl(
\sum_{\substack{\lambda=\gamma\cup \gamma'\in\Gamma_{3,4}(F_{ij}^{(m)}) \\ \gamma\in \Gamma_{4}(F_{ij}^{(m)}),\ \gamma'\in \Gamma_{3}(F_{ij}^{(m)}) \\ \overline{imj}\subset \gamma}} {\rm lk}\left(f(\lambda)\right)^{2}
+ \sum_{\substack{\lambda\in\Gamma_{3,3}(K_{n-1}^{(m)}) \\ \overline{ij}\not\subset \lambda}} {\rm lk}\left(f(\lambda)\right)^{2}
\Biggr)+ d. \nonumber
\end{eqnarray}
Let us take the sum of both sides of (\ref{eq1}) over $1\le i<j\le n$ and $i,j\neq m$. For an $n$-cycle $\gamma$ of $K_{n}$, let $i$ and $j$ be the two vertices of $K_{n}$ which are adjacent to $m$ in $\gamma$ ($1\le i<j\le n$ and $i,j\neq m$). Then $\gamma$ is an $n$-cycle of $F_{ij}^{(m)}$. This implies that 
\begin{eqnarray}\label{eq2}
\sum_{\substack{1\le i<j\le n \\ i,j\neq m}}\sum_{\gamma\in\Gamma_{n}(F_{ij}^{(m)})}a_{2}\left(f(\gamma)\right)
&=& \sum_{\gamma\in\Gamma_{n}(K_{n})}a_{2}\left(f(\gamma)\right).
\end{eqnarray}
For an $(n-1)$-cycle $\gamma$ of $K_{n-1}^{(m)}$, let $\overline{ij}$ be an edge of $K_{n-1}^{(m)}$ which is not contained in $\gamma$. Note that there are $\binom{n - 1}{2} - (n-1) = (n^{2} -5n +4)/2$ ways to choose such a pair of $i$ and $j$. This implies that 
\begin{eqnarray}\label{eq3}
\ \ \ \ \ \ \ \sum_{\substack{1\le i<j\le n \\ i,j\neq m}}\sum_{\substack{\gamma\in\Gamma_{n-1}(K_{n-1}^{(m)}) \\ \overline{ij}\not\subset \gamma}} a_{2}\left(f(\gamma)\right)
= \frac{n^{2} -5n +4}{2}\sum_{\gamma\in\Gamma_{n-1}(K_{n-1}^{(m)})}a_{2}\left(f(\gamma)\right).
\end{eqnarray}
For a $6$-cycle $\gamma$ of $K_{n}$ which contains the vertex $m$, let $i$ and $j$ be the two vertices of $K_{n}$ which are adjacent to $m$ in $\gamma$. Then $\gamma$ is a $6$-cycle of $F_{ij}^{(m)}$ which contains $\overline{imj}$. This implies that 
\begin{eqnarray}\label{eq4}
\sum_{\substack{1\le i<j\le n \\ i,j\neq m}}\sum_{\substack{\gamma\in\Gamma_{6}(F_{ij}^{(m)}) \\ \overline{imj}\subset \gamma}} a_{2}\left(f(\gamma)\right)
&=& \sum_{\substack{\gamma\in\Gamma_{6}(K_{n}) \\ m\subset \gamma}} a_{2}\left(f(\gamma)\right). 
\end{eqnarray}
For a $5$-cycle $\gamma$ of $K_{n-1}^{(m)}$, let $\overline{ij}$ be an edge of $K_{n-1}^{(m)}$ which is not contained in $\gamma$. Note that there are $\binom{n - 1}{2} - 5 = (n^{2} - 3n -8)/2$ ways to choose such a pair of $i$ and $j$. This implies that 
\begin{eqnarray}\label{eq5}
\sum_{\substack{1\le i<j\le n \\ i,j\neq m}}\sum_{\substack{\gamma\in\Gamma_{5}(K_{n-1}^{(m)}) \\ \overline{ij}\not\subset \gamma}} a_{2}\left(f(\gamma)\right)
&=& \frac{n^{2} - 3n -8}{2}\sum_{\gamma\in\Gamma_{5}(K_{n-1}^{(m)})}a_{2}\left(f(\gamma)\right).
\end{eqnarray}
For a pair of disjoint cycles $\lambda$ of $K_{n}$ consisting of a $4$-cycle $\gamma$ which contains the vertex $m$ and a $3$-cycle $\gamma'$, let $i$ and $j$ be the two vertices of $K_{n}$ which are adjacent to $m$ in $\gamma$. Then $\lambda$ is a pair of disjoint cycles of $K_{n}$ consisting of a $4$-cycle $\gamma$ which contains $\overline{imj}$ and a $3$-cycle $\gamma'$. This implies that 
\begin{eqnarray}\label{eq6}
\ \ \ \ \ \ \ \sum_{\substack{1\le i<j\le n \\ i,j\neq m}}\sum_{\substack{\lambda=\gamma\cup \gamma'\in\Gamma_{3,4}(F_{ij}^{(m)}) \\ \gamma\in \Gamma_{4}(F_{ij}^{(m)}),\ \gamma'\in \Gamma_{3}(F_{ij}^{(m)}) \\ \overline{imj}\subset \gamma}} {\rm lk}\left(f(\lambda)\right)^{2}
= \sum_{\substack{\lambda=\gamma\cup \gamma'\in\Gamma_{3,4}(K_{n}) \\ \gamma\in \Gamma_{4}(K_{n}),\ \gamma'\in \Gamma_{3}(K_{n}) \\ m\subset \gamma}} {\rm lk}\left(f(\lambda)\right)^{2}.
\end{eqnarray}
For a pair of disjoint $3$-cycles $\lambda$ of $K_{n-1}^{(m)}$, let $\overline{ij}$ be an edge of $K_{n-1}^{(m)}$ which is not contained in $\lambda$. Note that there are $\binom{n - 1}{2} - 6 =  (n^{2} -3n -10)/2$ ways to choose such a pair of $i$ and $j$. This implies that 
\begin{eqnarray}\label{eq7}
\ \ \ \ \sum_{\substack{1\le i<j\le n \\ i,j\neq m}}\sum_{\substack{\lambda\in\Gamma_{3,3}(K_{n-1}^{(m)}) \\ \overline{ij}\not\subset \lambda}} {\rm lk}\left(f(\lambda)\right)^{2}
= \frac{n^{2} -3n -10}{2}\sum_{\lambda\in\Gamma_{3,3}(K_{n-1}^{(m)})}{\rm lk}\left(f(\lambda)\right)^{2}.
\end{eqnarray}
By combining (\ref{eq2}), (\ref{eq3}), (\ref{eq4}), (\ref{eq5}), (\ref{eq6}) and (\ref{eq7}) with (\ref{eq1}), we have 
\begin{eqnarray}
&&\sum_{\gamma\in\Gamma_{n}\left(K_{n}\right)}a_{2}\left(f(\gamma)\right)
+ \frac{n^{2} - 5n +4}{2}\sum_{\gamma\in\Gamma_{n-1}(K_{n-1}^{(m)})}a_{2}\left(f(\gamma)\right)\label{eq8}\\
&&+b\Biggl(
\sum_{\substack{\gamma\in\Gamma_{6}\left(K_{n}\right) \\ m\subset \gamma}} a_{2}\left(f(\gamma)\right)
+ \frac{n^{2} -3n -8}{2} \sum_{\gamma\in\Gamma_{5}(K_{n-1}^{(m)})}a_{2}\left(f(\gamma)\right)
\Biggr)\nonumber\\
&=& 
c\Biggl(
\sum_{\substack{\lambda=\gamma\cup \gamma'\in\Gamma_{3,4}\left(K_{n}\right) \\ \gamma\in \Gamma_{4}\left(K_{n}\right),\ \gamma'\in \Gamma_{3}\left(K_{n}\right) \\ m\subset \gamma}} {\rm lk}\left(f(\lambda)\right)^{2} + \frac{n^{2} - 3n -10}{2} \sum_{\lambda\in\Gamma_{3,3}(K_{n-1}^{(m)})}{\rm lk}\left(f(\lambda)\right)^{2}
\Biggr)\nonumber\\
&& + \frac{d(n^{2} - 3n +2)}{2}.  \nonumber
\end{eqnarray}
Then for the embedding $f$ restricted to $K_{n-1}^{(m)}$, by the assumption we have  
\begin{eqnarray}\label{eq9}
&& \sum_{\gamma\in \Gamma_{n-1}(K_{n-1}^{(m)})}a_{2}\left(f(\gamma)\right)\\
&=& -b\sum_{\gamma\in \Gamma_{5}(K_{n-1}^{(m)})}a_{2}\left(f(\gamma)\right) 
+ 
c\sum_{\lambda\in \Gamma_{3,3}(K_{n-1}^{(m)})}{{\rm lk}\left(f(\lambda)\right)}^{2} + d. \nonumber  
\end{eqnarray}
By combining (\ref{eq8}) and (\ref{eq9}), we have 
\begin{eqnarray}\label{eq10}
&&\sum_{\gamma\in\Gamma_{n}\left(K_{n}\right)}a_{2}\left(f(\gamma)\right)
+ b\sum_{\substack{\gamma\in\Gamma_{6}\left(K_{n}\right) \\ m\subset \gamma}} a_{2}\left(f(\gamma)\right)
+ b(n-6) \sum_{\gamma\in\Gamma_{5}(K_{n-1}^{(m)})}a_{2}\left(f(\gamma)\right)\\
&=& 
c\sum_{\substack{\lambda=\gamma\cup \gamma'\in\Gamma_{3,4}\left(K_{n}\right) \\ \gamma\in \Gamma_{4}\left(K_{n}\right),\ \gamma'\in \Gamma_{3}\left(K_{n}\right) \\ m\subset \gamma}} {\rm lk}\left(f(\lambda)\right)^{2} + 
c(n-7)\sum_{\lambda\in\Gamma_{3,3}(K_{n-1}^{(m)})}{\rm lk}\left(f(\lambda)\right)^{2}
+ d(n-1).  \nonumber
\end{eqnarray}
Now we take the sum of both sides of (\ref{eq10}) over $m = 1,2,\ldots,n$. For a $6$-cycle $\gamma$ of $K_{n}$, let $m$ be a vertex of $K_{n}$ which is contained in $\gamma$. Note that there are six ways to choose such a vertex $m$. This implies that 
\begin{eqnarray}\label{eq11}
\sum_{m=1}^{n}\sum_{\substack{\gamma\in\Gamma_{6}\left(K_{n}\right) \\ m\subset \gamma}} a_{2}\left(f(\gamma)\right)
&=& 6\sum_{\gamma\in \Gamma_{6}\left(K_{n}\right)}a_{2}\left(f(\gamma)\right).
\end{eqnarray}
For a $5$-cycle $\gamma$ of $K_{n}$, let $m$ be a vertex of $K_{n}$ which is not contained in $\gamma$. Then $\gamma$ is a $5$-cycle of $K_{n-1}^{(m)}$. Note that there are $n-5$ ways to choose such a vertex $m$. This implies that 
\begin{eqnarray}\label{eq12}
\sum_{m=1}^{n}\sum_{\gamma\in\Gamma_{5}(K_{n-1}^{(m)})}a_{2}\left(f(\gamma)\right)
&=& (n-5)\sum_{\gamma\in \Gamma_{5}\left(K_{n}\right)}a_{2}\left(f(\gamma)\right).
\end{eqnarray}
For a pair of disjoint cycles $\lambda$ of $K_{n}$ consisting of a $4$-cycle $\gamma$ and a $3$-cycle $\gamma'$, let $m$ be a vertex of $K_{n}$ which is contained in $\gamma$. Note that there are four ways to choose such a vertex $m$. This implies that 
\begin{eqnarray}\label{eq13}
\sum_{m=1}^{n}\sum_{\substack{\lambda=\gamma\cup \gamma'\in\Gamma_{3,4}\left(K_{n}\right) \\ \gamma\in \Gamma_{4}\left(K_{n}\right),\ \gamma'\in \Gamma_{3}\left(K_{n}\right) \\ m\subset \gamma}} {\rm lk}\left(f(\lambda)\right)^{2}
&=& 4\sum_{\lambda\in \Gamma_{3,4}\left(K_{n}\right)}{\rm lk}\left(f(\lambda)\right)^{2}. 
\end{eqnarray}
For a pair of two disjoint $3$-cycles $\lambda$ of $K_{n}$, let $m$ be a vertex of $K_{n}$ which is not contained in $\lambda$. Then $\lambda$ is a pair of two disjoint $3$-cycles of $K_{n-1}^{(m)}$. Note that there are $n-6$ ways to choose such a vertex $m$. This implies that 
\begin{eqnarray}\label{eq14}
\sum_{m=1}^{n}\sum_{\lambda\in\Gamma_{3,3}(K_{n-1}^{(m)})}{\rm lk}\left(f(\lambda)\right)^{2}
&=& (n-6)\sum_{\gamma\in \Gamma_{3,3}\left(K_{n}\right)}{\rm lk}\left(f(\lambda)\right)^{2}.
\end{eqnarray}
By combining  (\ref{eq11}), (\ref{eq12}), (\ref{eq13}) and (\ref{eq14}) with (\ref{eq10}), we have 
\begin{eqnarray}\label{eq15}
&&n\sum_{\gamma\in\Gamma_{n}\left(K_{n}\right)}a_{2}\left(f(\gamma)\right)
+ 6b\sum_{\gamma\in\Gamma_{6}\left(K_{n}\right)}a_{2}\left(f(\gamma)\right)\\
&& + b(n-5)(n-6)\sum_{\gamma\in\Gamma_{5}\left(K_{n}\right)}a_{2}\left(f(\gamma)\right) \nonumber\\
&=& 
4c\sum_{\lambda\in\Gamma_{3,4}\left(K_{n}\right)}{\rm lk}\left(f(\lambda)\right)^{2}
+c(n-6)(n-7)\sum_{\lambda\in\Gamma_{3,3}\left(K_{n}\right)}{\rm lk}\left(f(\lambda)\right)^{2}
+ dn(n-1).  \nonumber
\end{eqnarray}
Then by (\ref{eq15}) and Lemma \ref{nlem} (1) and (2), we have the desired conclusion.  
\end{proof}

\begin{proof}[Proof of Theorem \ref{mainthm}]
We prove this by induction on $n$. In the case of $n = 6$, by Theorem \ref{main1} we have the result. Assume that $n \ge 7$, then we have 
\begin{eqnarray}\label{ind0}
&&\sum_{\gamma\in \Gamma_{n}\left(K_{n-1}\right)}a_{2}\left(g(\gamma)\right)
- (n-6)!\sum_{\gamma\in \Gamma_{5}\left(K_{n-1}\right)}a_{2}\left(g(\gamma)\right)\\
&=& 
\frac{(n-6)!}{2} 
\sum_{\lambda\in \Gamma_{3,3}\left(K_{n-1}\right)}{{\rm lk}\left(g(\lambda)\right)}^{2}
- \frac{(n-6)!}{2}\binom{n-2}{5}\nonumber
\end{eqnarray}
for any spatial embedding $g$ of $K_{n-1}$. Then by (\ref{ind0}) and Lemma \ref{ind}, we have  
\begin{eqnarray*}
&& \sum_{\gamma\in \Gamma_{n}\left(K_{n}\right)}a_{2}\left(f(\gamma)\right)
-(n - 5)! \sum_{\gamma\in \Gamma_{5}\left(K_{n}\right)}a_{2}\left(f(\gamma)\right)\\ 
&=&   
\frac{1}{n}\left(\frac{(n - 6)!}{2}(n - 6)(n + 1) + 3(n - 6)!\right)\sum_{\lambda\in \Gamma_{3,3}\left(K_{n}\right)}{{\rm lk}\left(f(\lambda)\right)}^{2} \\
&& -\frac{(n-6)!}{2}\binom{n-2}{5}(n-1) - \frac{(n - 6)!}{2}\binom{n - 1}{5}\\
&=& \frac{(n - 5)!}{2}\sum_{\lambda\in \Gamma_{3,3}\left(K_{n}\right)}{{\rm lk}\left(f(\lambda)\right)}^{2} 
- \frac{(n-5)!}{2}\binom{n-1}{5} 
\end{eqnarray*}
for any spatial embedding $f$ of $K_{n}$. This completes the proof. 
\end{proof}

\begin{proof}[Proof of Corollary \ref{maincor}]
Note that  no pair of two disjoint $3$-cycles $\lambda$ of $K_{n}$ is shared by two  distinct subgraphs of $K_{n}$ isomorphic to $K_{6}$. Then Theorem \ref{CG1} (1) implies that $\sum_{\lambda\in \Gamma_{3,3}\left(K_{n}\right)}{{\rm lk}\left(f(\lambda)\right)}^{2}$ is greater than or equal to the number of subgraphs of $K_{n}$ isomorphic  to $K_{6}$, that is equal to $\binom{n}{6}$, and by a direct calculation  we have 
\begin{eqnarray}\label{dcal}
\frac{(n-5)!}{2} 
\bigg(
\binom{n}{6}
- \binom{n-1}{5}
\bigg)
= 
\frac{(n-5)(n-6)(n-1)!}{2\cdot 6!}.
\end{eqnarray}
Thus by (\ref{dcal}) and Theorem \ref{mainthm}, we have the result. 
\end{proof}

\begin{Remark}\label{mainrem0}
Endo-Otsuki introduced a certain special spatial embedding $f_{\rm b}$ of  $K_{n}$, a {\it canonical book presentation} of $K_{n}$ \cite{EO94}, and Otsuki also showed that $f_{\rm b}\left(K_{n}\right)$ contains exactly $\binom{n}{6}$ Hopf links corresponding to all the pairs of two disjoint $3$-cycles of $K_{n}$ if $n\ge 6$ \cite{Otsuki96}. Thus the lower bound of Corollary \ref{maincor} is sharp.  Furthermore, for any $5$-cycle $\gamma$ of $K_{n}$, $f_{\rm b}(\gamma)$ is a trivial knot. Thus for an integer $n\ge 6$, we have 
\begin{eqnarray*}
\sum_{\gamma\in \Gamma_{n}\left(K_{n}\right)}a_{2}\left(f_{\rm b}(\gamma)\right)
= 
\frac{(n-5)(n-6)(n-1)!}{2\cdot 6!}.
\end{eqnarray*}
\end{Remark}

\begin{proof}[Proof of Corollary \ref{maincor0}]
For any two spatial embeddings $f$ and $g$ of $K_{n}$, by Theorem \ref{mainthm}, we have 
\begin{eqnarray}\label{co1}
&& \sum_{\gamma\in \Gamma_{n}\left(K_{n}\right)}a_{2}\left(f(\gamma)\right)
- \sum_{\gamma\in \Gamma_{n}\left(K_{n}\right)}a_{2}\left(g(\gamma)\right)\\ 
&\equiv& 
\frac{(n-5)!}{2}
\Big(\sum_{\lambda\in \Gamma_{3,3}\left(K_{n}\right)}{{\rm lk}\left(f(\lambda)\right)}^{2}
- \sum_{\lambda\in \Gamma_{3,3}\left(K_{n}\right)}{{\rm lk}\left(g(\lambda)\right)}^{2}
\Big)\pmod{(n-5)!}. \nonumber
\end{eqnarray}
Since  $\sum_{\lambda\in \Gamma_{3,3}\left(K_{n}\right)}{{\rm lk}\left(f(\lambda)\right)}^{2}$ and $\sum_{\lambda\in \Gamma_{3,3}\left(K_{n}\right)}{{\rm lk}\left(g(\lambda)\right)}^{2}$ have the same parity, that is also equal to the parity of $\binom{n}{6}$, by (\ref{co1}), we have 
\begin{eqnarray}\label{co2}
\sum_{\gamma\in \Gamma_{n}\left(K_{n}\right)}a_{2}\left(f(\gamma)\right)
\equiv \sum_{\gamma\in \Gamma_{n}\left(K_{n}\right)}a_{2}\left(g(\gamma)\right)
\pmod{(n-5)!}. 
\end{eqnarray}
Note that  there exists a spatial embedding $g$ of $K_{n}$ such that 
\begin{eqnarray}\label{co4}
\sum_{\lambda\in \Gamma_{3,3}\left(K_{n}\right)}{{\rm lk}\left(g(\lambda)\right)}^{2} = \binom{n}{6}, 
\end{eqnarray}
see Remark \ref{mainrem0} or Remark \ref{recrec}. Thus by (\ref{co2}) and (\ref{co4}), we have 
\begin{eqnarray}\label{co5}
\sum_{\gamma\in \Gamma_{n}\left(K_{n}\right)}a_{2}\left(f(\gamma)\right)
\equiv  
\frac{(n-5)!}{2} 
\bigg(
\binom{n}{6}
- \binom{n-1}{5}
\bigg) \pmod{(n-5)!}
\end{eqnarray}
for any spatial embedding $f$ of $K_{n}$. Here, it can be seen that $\binom{n}{6}$ is odd if and only if $n\equiv 6,7\pmod{8}$, and $\binom{n-1}{5}$ is odd if and only if  $n\equiv 0,6\pmod{8}$ by an application of Lucas's theorem for binomial coefficients (see \cite{Fine47} for example). If $n\not\equiv 0,7\pmod{8}$, then since $\binom{n}{6} - \binom{n-1}{5}$ is even, by (\ref{co5}), we have 
\begin{eqnarray*}
\sum_{\gamma\in \Gamma_{n}\left(K_{n}\right)}a_{2}\left(f(\gamma)\right)
\equiv 0\pmod{(n-5)!}. 
\end{eqnarray*}
If $n\equiv 0\pmod{8}$, then since $\binom{n}{6}$ is even and  $\binom{n-1}{5}$ is odd, by (\ref{co5}), we have 
\begin{eqnarray*}
\sum_{\gamma\in \Gamma_{n}\left(K_{n}\right)}a_{2}\left(f(\gamma)\right)
\equiv - \frac{(n-5)!}{2} \binom{n-1}{5}\pmod{(n-5)!}. 
\end{eqnarray*}
If $n\equiv 7\pmod{8}$, then since $\binom{n}{6}$ is odd and $\binom{n-1}{5}$ is even, by (\ref{co5}), we have 
\begin{eqnarray*}
\sum_{\gamma\in \Gamma_{n}\left(K_{n}\right)}a_{2}\left(f(\gamma)\right)
\equiv \frac{(n-5)!}{2}\binom{n}{6}\pmod{(n-5)!}. 
\end{eqnarray*}
This completes the proof. 
\end{proof}

\begin{Remark}\label{mainrem1}
By applying the case of $n = 7$ in Corollary \ref{maincor0}, we have 
\begin{eqnarray*}
\sum_{\gamma\in \Gamma_{7}\left(K_{7}\right)}a_{2}\left(f(\gamma)\right)
\equiv \frac{2!}{2}\binom{7}{6}
\equiv 1
\pmod{2}, 
\end{eqnarray*}
that is, Theorem \ref{CG1} (2). On the other hand, for any spatial embedding $f$ of $K_{8}$,  it was shown that $\sum_{\gamma\in \Gamma_{8}\left(K_{8}\right)}a_{2}\left(f(\gamma)\right)\equiv 0\pmod{3}$ by Foisy \cite{F08} and $1\pmod{2}$ by Hirano \cite{HiranoD}. These  results imply that $\sum_{\gamma\in \Gamma_{8}\left(K_{8}\right)}a_{2}\left(f(\gamma)\right)\equiv 3\pmod{6}$, and it can also be shown by applying the case of $n = 8$ in Corollary \ref{maincor0}: 
\begin{eqnarray*}
\sum_{\gamma\in \Gamma_{8}\left(K_{8}\right)}a_{2}\left(f(\gamma)\right)
\equiv -\frac{3!}{2}\binom{7}{5}
= -63 
\equiv 3 \pmod{6}. 
\end{eqnarray*}
Hirano also showed that $\sum_{\gamma\in \Gamma_{n}\left(K_{n}\right)}a_{2}\left(f(\gamma)\right)\equiv 0\pmod{2}$ for any spatial embedding $f$ of $K_{n}$ if  $n\ge 9$ \cite{HiranoD}. Corollary \ref{maincor0} also generalizes it remarkably. 
\end{Remark}

\begin{proof}[Proof of Corollary \ref{maincor2}]
We obtain the desired lower bound from Corollary \ref{maincor} directly, since for every $5$-cycle $\gamma$, $f_{\rm r}(\gamma)$ is trivial. On the other hand, it is known that every rectilinear spatial graph of $K_{6}$ contains at most three Hopf links (Hughes \cite{Hughes06}, Huh-Jeon \cite{HJ07}, Nikkuni \cite{Nikkuni09}). This implies that $\sum_{\lambda\in \Gamma_{3,3}\left(K_{n}\right)}{{\rm lk}\left(f_{\rm r}(\lambda)\right)}^{2}$ is less than or equal to $3\binom{n}{6}$, and by a direct calculation we have 
\begin{eqnarray}\label{dcal2}
\frac{(n-5)!}{2} 
\bigg(
3\binom{n}{6}
- \binom{n-1}{5}
\bigg)
= 
\frac{3(n-2)(n-5)(n-1)!}{2\cdot 6!}.
\end{eqnarray}
Thus by (\ref{dcal2}) and Theorem \ref{mainthmrecti}, we get the desired upper bound.  
\end{proof}

\begin{figure}[htbp]
      \begin{center}
\scalebox{0.465}{\includegraphics*{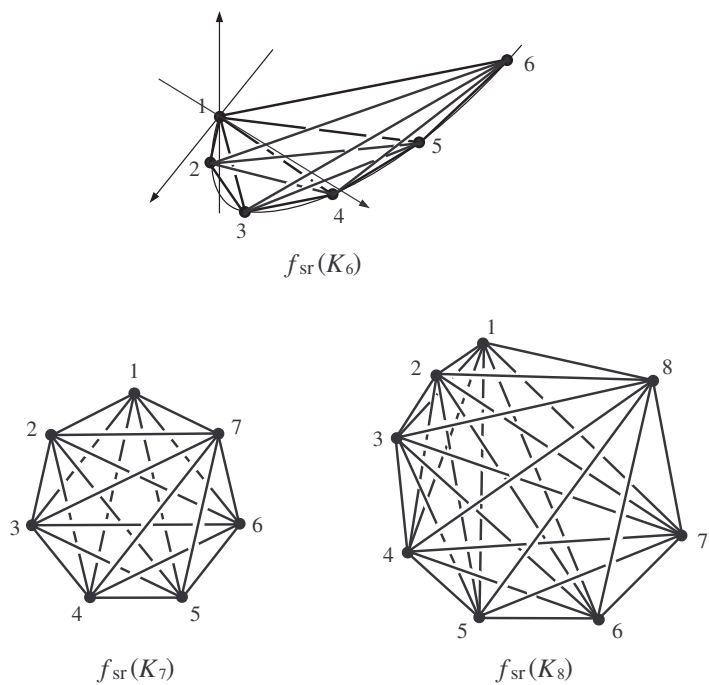}}
      \end{center}
   \caption{Standard rectilinear spatial embedding $f_{\rm sr}$ of $K_{n}\ (n = 6,7,8)$}
  \label{K678rect}
\end{figure} 

\begin{Remark}\label{recrec}
A special rectilinear spatial embedding $f_{\rm sr}$ of $K_{n}$ can be constructed by taking $n$ vertices $1,2,\ldots,n$ of $K_{n}$ in order on the moment curve $(t,t^{2},t^{3})$ in ${\mathbb R}^{3}$ and connecting every pair of two distinct vertices $i$ and $j$ by a straight line segment, see Fig. \ref{K678rect} for $n=6,7,8$. We call $f_{\rm sr}$ the {\it standard rectilinear spatial embedding} of $K_{n}$. For the standard rectilinear spatial embedding $f_{\rm sr}$ of $K_{n}$ ($n\ge 6$) and a subgraph $F$ of $K_{n}$ isomorphic to $K_{6}$, it can be easily seen that the embedding $f_{\rm sr}$ restricted to $F$ is equivalent to the standard rectlinear spatial embedding of $K_{6}$. Since the standard rectilinear spatial graph of $K_{6}$ contains exactly one nonsplittable $2$-component link which is a Hopf link, $f_{\rm sr}\left(K_{n}\right)$ contains exactly $\binom{n}{6}$ triangle-triangle Hopf links. Thus the lower bound in Corollary \ref{maincor2} is sharp. 
\end{Remark}

Before proving Corollary \ref{maincor3}, we recall two geometric invariants of knots and links. For a knot or link $L$, 
the {\it crossing number} of $L$ is the minimum number of crossings in a regular diagram of $L$ on the plane, denoted by $c(L)$, and the {\it stick number} of $L$ is the minimum number of edges in a polygon which represents $L$, denoted by $s(L)$.

\begin{proof}[Proof of Corollary \ref{maincor3}]
For a knot $K$, it has been shown that 
\begin{eqnarray}\label{cal}
c(K) \le \frac{\left(s(K) - 3\right)\left(s(K) - 4\right)}{2}
\end{eqnarray}
by Calvo \cite[Theorem 4]{Calvo01}, and also has been shown that 
\begin{eqnarray}\label{pv1e}
a_{2}(K) \le \frac{c(K)^{2}}{8}
\end{eqnarray}
by Polyak-Viro \cite[Theorem 1.E]{PV01}. By combining (\ref{cal}) and (\ref{pv1e}), for a polygonal knot $K$ with less than or equal to $n$ sticks, we have 
\begin{eqnarray}\label{a2esti} 
a_{2}(K) \le 
\left\lfloor 
\frac{\mathstrut (n-3)^{2}(n-4)^{2}}{\mathstrut 32}
\right\rfloor. 
\end{eqnarray}
Then by the lower bound in Corollary \ref{maincor2} and (\ref{a2esti}), we have the desired estimation from below. 
\end{proof}

\begin{Remark}\label{fh}
The concrete values of $r_{n}$ for $7\le n\le 15$ are given in the following table. Note that in the case of $n=8$, we can obtain an estimate from below better than $r_{8}$ of the number of nontrivial Hamiltonian knots with a positive value of $a_{2}$ in every rectilinear spatial graph of $K_{8}$, see Example \ref{ub2} and Remark \ref{3311}. 

\begin{table}[htbp]
\begin{center}
  \begin{tabular}{|c||c|c|c|c|c|c|c|c|c|c|} \hline
    $n$ & $7$ & $8$ & $9$ & $10$ & $11$ & $12$ & $13$ & $14$  & $15$  &$\cdots$ \\ \hline
    $r_{n}$ & $1$ & $2$ & $12$ & $92$ & $772$ & $7187$ & $73628$ & $823680$ &  $10015889$ & $\cdots$ \\ \hline
  \end{tabular}
\end{center}
\end{table}
\end{Remark}

\section{Examples and Problems} 

In the following examples, we denote a $k$-cycle $\overline{i_{1}i_{2}}\cup \overline{i_{2}i_{3}}\cup \cdots \cup \overline{i_{k-1}i_{k}}\cup \overline{i_{k}i_{1}}$ of $K_{n}$ by $[i_{1}i_{2}\cdots i_{k}]$. We also recall the following fundamental results on stick numbers for knots and links (see Adams \cite[\S 1.6]{Adams04}, Negami \cite[Theorem 6]{Negami91}, Adams-Brennan-Greilsheimer-Woo \cite[Theorem 2.1]{ABGW97} and Calvo \cite[Theorem 1]{Calvo01}), where we denote each of knots and links appearing in the statement by using its label in Rolfsen's table \cite{R76}. 

\begin{Proposition}\label{stick} Let $L$ be a link. Then the following statements hold. 
\begin{enumerate}
\item If $L$ is a nontrivial knot, then $s(L)\ge 6$. 
\item $s(L)=6$ if and only if $L$ is equivalent to $3_{1}$, $0_{1}^{2}$ or $2_{1}^{2}$.  
\item $s(L)=7$ if and only if $L $ is equivalent to $4_{1}$ or $4_{1}^{2}$. 
\item $s(L)=8$ if and only if $L$ is equivalent to $5_{1}$, $5_{2}$, $6_{1}$, $6_{2}$, $6_{3}$,  the granny knot $3_{1}\# 3_{1}$, the square knot $3_{1}\# 3_{1}^{*}$, $8_{19}$, $8_{20}$ or $5_{1}^{2}$.
\end{enumerate}
\end{Proposition}

\begin{Example}\label{ub0}
Let $f_{\rm r}$ be a rectilinear spatial embedding of  $K_{6}$. Then by Theorem \ref{mainthmrecti} (Theorem \ref{main1}) and Corollary \ref{maincor2}, we have 
\begin{eqnarray}\label{hj60}
\sum_{\gamma\in \Gamma_{6}\left(K_{6}\right)}a_{2}\left(f_{\rm r}(\gamma)\right)
= 
\frac{1}{2}\sum_{\lambda\in \Gamma_{3,3}\left(K_{6}\right)}{{\rm lk}\left(f_{\rm r}(\lambda)\right)}^{2}
- \frac{1}{2}, 
\end{eqnarray}
\begin{eqnarray}\label{hj6}
0 \le \sum_{\gamma\in \Gamma_{6}\left(K_{6}\right)}a_{2}\left(f_{\rm r}(\gamma)\right) 
\le 1. 
\end{eqnarray}
As it has been shown in \cite[\S 4]{Nikkuni09}, (\ref{hj60}) and (\ref{hj6})  enable us to give an alternative topological proof of the fact that every rectilinear spatial graph $f_{\rm r}\left(K_{6}\right)$ contains at most one trefoil knot, in particular,  $f_{\rm r}\left(K_{6}\right)$ does not contain a trefoil knot if and only if $f_{\rm r}\left(K_{6}\right)$ contains exactly one Hopf link, and $f_{\rm r}\left(K_{6}\right)$ contains a trefoil knot if and only if $f_{\rm r}\left(K_{6}\right)$ contains exactly three Hopf links, which was originally proven by Huh-Jeon \cite{HJ07} in combinatorial way. Actually, it follows from Proposition \ref{stick} (1) and (2) that $\sum_{\gamma\in \Gamma_{6}\left(K_{6}\right)}a_{2}\left(f_{\rm r}(\gamma)\right)$ equals the number of trefoil knots in $f_{\rm r}\left(K_{6}\right)$ because $a_{2}(3_{1}) = 1$, and $\sum_{\lambda\in \Gamma_{3,3}\left(K_{6}\right)}{{\rm lk}\left(f_{\rm r}(\lambda)\right)}^{2}$ equals  the number of Hopf links in $f_{\rm r}\left(K_{6}\right)$. 
\end{Example}

\begin{Example}\label{ub}
For a spatial embedding $f$ of $K_{7}$, by Theorem \ref{mainthm}, we have 
\begin{eqnarray}\label{new1}
\sum_{\gamma\in \Gamma_{7}\left(K_{7}\right)}a_{2}\left(f(\gamma)\right)
- 2\sum_{\gamma\in \Gamma_{5}\left(K_{7}\right)}a_{2}\left(f(\gamma)\right)
= 
\sum_{\lambda\in \Gamma_{3,3}\left(K_{7}\right)}{{\rm lk}\left(f(\lambda)\right)}^{2}
- 6, 
\end{eqnarray}
and we also have 
\begin{eqnarray}\label{e7}
\sum_{\gamma\in \Gamma_{7}\left(K_{7}\right)}a_{2}\left(f(\gamma)\right)
- 2\sum_{\gamma\in \Gamma_{5}\left(K_{7}\right)}a_{2}\left(f(\gamma)\right)
\ge 1, 
\end{eqnarray}
\begin{eqnarray*}
\sum_{\gamma\in \Gamma_{7}\left(K_{7}\right)}a_{2}\left(f(\gamma)\right)\equiv 1\pmod{2} 
\end{eqnarray*}
by Corollaries \ref{maincor} and \ref{maincor0}. Let $h$ be the spatial embedding of $K_{7}$ as illustrated in the left hand side of Fig. \ref{K78examples}. It is known that $h\left(K_{7}\right)$ contains exactly one nontrivial knot $h([1357246])$ which is a trefoil knot \cite{CG83}. Since $a_{2}\left(3_{1}\right) = 1$, the embedding $h$ realizes the lower bound in  (\ref{e7}). In particular for a rectilinear spatial embedding $f_{\rm r}$ of $K_{7}$, by Theorem \ref{mainthmrecti} and Corollary \ref{maincor2}, we have 
\begin{eqnarray}\label{e7200}
\sum_{\gamma\in \Gamma_{7}\left(K_{7}\right)}a_{2}\left(f_{\rm r}(\gamma)\right)
= 
\sum_{\lambda\in \Gamma_{3,3}\left(K_{7}\right)}{{\rm lk}\left(f_{\rm r}(\lambda)\right)}^{2}
- 6, 
\end{eqnarray}
\begin{eqnarray}\label{e72}
1 \le \sum_{\gamma\in \Gamma_{7}\left(K_{7}\right)}a_{2}\left(f_{\rm r}(\gamma)\right) 
\le 15. 
\end{eqnarray}
As it has been shown in \cite{Nikkuni09}, the lower bound in (\ref{e72}) enables us to give much simpler  topological proof of the fact that every rectilinear spatial graph of $K_{7}$ contains a trefoil knot, which was originally proven by Brown \cite{B77} and Ram{\'\i}rez Alfons{\'\i}n \cite{RA99} in combinatorial and computational way. Actually, by (\ref{e72}), there exists at least one Hamiltonian cycle $\gamma_{0}$ of $K_{7}$ such that $a_{2}\left(f_{\rm r}(\gamma_{0})\right) > 0$. Then by Proposition \ref{stick} (2) and (3), $f_{\rm r}(\gamma_{0})$ is either a trefoil knot or a figure eight knot. Since $a_{2}\left(4_{1}\right)  = - 1$, the knot $f_{\rm r}\left(\gamma_{0}\right)$ must be a trefoil knot. We also remark here that $h$ is equivalnt to the standard rectilinear spatial embedding $f_{\rm sr}$ of $K_{7}$ in Fig. \ref{K678rect}. We refer the reader to \cite{Huh12}, \cite{L10} for related works on rectilinear spatial graphs of $K_{7}$ (especially in \cite{Huh12}, a remarkable result is shown that the number of figure eight knots in a rectilinear spatial graph of $K_{7}$ is at most three). Moreover, according to a computer search in \cite{Jeon10}, there seems to be no rectilinear  embedding $f_{\rm r}$ of $K_{7}$ such that $\sum_{\gamma\in \Gamma_{7}\left(K_{7}\right)}a_{2}\left(f_{\rm r}(\gamma)\right) = 13, 15$, or equivalently by (\ref{e7200}), $\sum_{\lambda\in \Gamma_{3,3}(K_{7})}{\rm lk}(f_{\rm r}(\gamma))^{2} = 19,21$. This strongly suggests that the upper bound in Corollary \ref{maincor2} is not sharp.  
\end{Example}

\begin{figure}[htbp]
      \begin{center}
\scalebox{0.525}{\includegraphics*{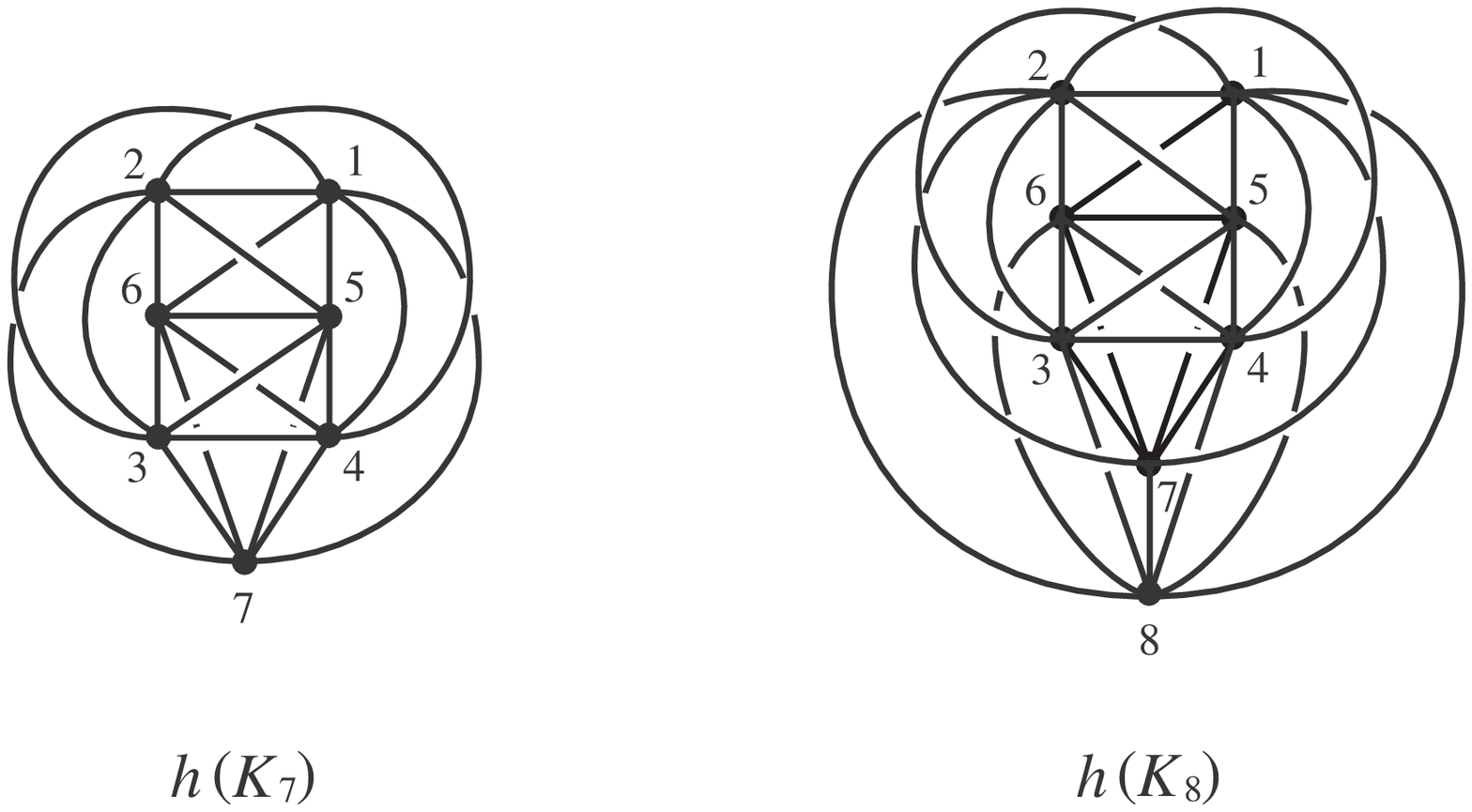}}
      \end{center}
   \caption{}
  \label{K78examples}
\end{figure} 

\begin{Problem}\label{prob1}
Determine the sharp upper bound of $\sum_{\gamma\in \Gamma_{n}\left(K_{n}\right)}a_{2}\left(f_{\rm r}(\gamma)\right)$ for all rectilinear spatial embeddings $f_{\rm r}$ of $K_{n}$ for each $n\ge 7$.  
\end{Problem}

\begin{Example}\label{ub2}
For a spatial embedding $f$ of $K_{8}$, by Theorem \ref{mainthm}, we have 
\begin{eqnarray}\label{new2}
\ \ \ \ \ \sum_{\gamma\in \Gamma_{8}\left(K_{8}\right)}a_{2}\left(f(\gamma)\right)
- 6\sum_{\gamma\in \Gamma_{5}\left(K_{8}\right)}a_{2}\left(f(\gamma)\right)
= 
3\sum_{\lambda\in \Gamma_{3,3}\left(K_{8}\right)}{{\rm lk}\left(f(\lambda)\right)}^{2}
- 63, 
\end{eqnarray}
and we also have 
\begin{eqnarray}\label{e8}
&& \sum_{\gamma\in \Gamma_{8}\left(K_{8}\right)}a_{2}\left(f(\gamma)\right)
- 6\sum_{\gamma\in \Gamma_{5}\left(K_{8}\right)}a_{2}\left(f(\gamma)\right)
\ge 21,
\end{eqnarray}
\begin{eqnarray*}
\sum_{\gamma\in \Gamma_{8}\left(K_{8}\right)}a_{2}\left(f(\gamma)\right) 
\equiv 3\pmod{6} 
\end{eqnarray*}
by Corollaries \ref{maincor} and \ref{maincor0}. Let $h$ be the spatial embedding of $K_{8}$ as illustrated in the right hand side of Fig. \ref{K78examples}. It is known that $h\left(K_{8}\right)$ contains exactly twenty one nontrivial Hamiltonian knots, all of which are trefoil knots \cite{BBFHL07}. Since $h(\gamma)$ is a trivial knot for any $5$-cycle $\gamma$ of $K_{8}$, the embedding $h$ realizes the lower bound in (\ref{e8}). 
In particular for a rectilinear spatial embedding $f_{\rm r}$ of $K_{8}$, by Theorem \ref{mainthmrecti} and Corollary \ref{maincor2}, we have 
\begin{eqnarray}\label{new2r}
\sum_{\gamma\in \Gamma_{8}\left(K_{8}\right)}a_{2}\left(f_{\rm r}(\gamma)\right)
= 
3\sum_{\lambda\in \Gamma_{3,3}\left(K_{8}\right)}{{\rm lk}\left(f_{\rm r}(\lambda)\right)}^{2}
- 63, 
\end{eqnarray}
\begin{eqnarray}\label{e82}
21\le \sum_{\gamma\in \Gamma_{8}\left(K_{8}\right)}a_{2}\left(f_{\rm r}(\gamma)\right) 
\le 189. 
\end{eqnarray}
By Proposition \ref{stick}, all of the polygonal knots with eight sticks are $0_{1}$, $3_{1}$, $4_{1}$, $5_{1}$, $5_{2}$, $6_{1}$, $6_{2}$, $6_{3}$, $3_{1}\# 3_{1}$, $3_{1}\# 3_{1}^{*}$, $8_{19}$ and $8_{20}$. Moreover, the values of $a_{2}$ for them are as follows: 

\begin{table}[htbp]
\begin{center}
  \begin{tabular}{|c||c|c|c|c|c|c|c|c|c|c|c|c|} \hline
 $K$    & $0_{1}$ & $3_{1}$ & $4_{1}$ & $5_{1}$ & $5_{2}$ & $6_{1}$ & $6_{2}$ & $6_{3}$  & $3_{1}\# 3_{1}$  & $3_{1}\# 3_{1}^{*}$ & $8_{19}$ & $8_{20}$\\ \hline
    $a_{2}(K)$ & $0$ & $1$ & $-1$ & $3$ & $2$ & $-2$ & $-1$ & $1$ & $2$ & $2$ & $5$ & $2$ \\ \hline
  \end{tabular}
\end{center}
\end{table}

\noindent
Thus it follows from (\ref{e82}) that every rectilinear spatial graph of $K_{8}$ always contains at least one of  $3_{1}$, $5_{1}$, $5_{2}$, $6_{3}$,  $3_{1}\# 3_{1}$, $3_{1}\# 3_{1}^{*}$, $8_{19}$ and  $8_{20}$ as a  Hamiltonian knot. Moreover, since the maximum value of $a_{2}$ in  every polygonal knot with exactly eight sticks is equal to five, we can refine (\ref{a2esti}) if $n = 8$ and then we can also refine Corollary \ref{maincor3}: the minimum number of nontrivial Hamiltonian knots with a positive value of $a_{2}$ in every rectilinear spatial graph of $K_{8}$ is at least $\lceil 21/5\rceil = 5$.  But this is not yet the sharp lower bound, see Remark \ref{3311}. 

As we mentioned in Remark \ref{recrec}, the standard rectilinear spatial embedding $f_{\rm sr}$ of $K_{8}$ in Fig. \ref{K678rect} realizes the lower bound in (\ref{e82}).  Moreover, it is known that all of the nontrivial Hamiltonian knots in $f_{\rm sr}\left(K_{8}\right)$ are trefoil knots \cite{RA08}. This means that $f_{\rm sr}\left(K_{8}\right)$ also contains exactly twenty one nontrivial Hamiltonian knots, all of which are trefoil knots. We also remark here that $h$ and $f_{\rm sr}$ are not equivalent 
because  $h\left(K_{8}\right)$ contains a ``triangle-pentagon'' link with nonzero even linking number (actually $h([257]\cup [13846])$ is equivalent to $4_{1}^{2}$), but $f_{\rm sr}\left(K_{8}\right)$ does not contain such a triangle-pentagon link. The authors do not know whether the embedding $h$ is equivalent to a certain rectilinear spatial embedding of $K_{8}$ or not.

\end{Example}

\begin{Remark}\label{3311}
It is known that every rectilinear spatial graph of $K_{3,3,1,1}$ contains at least one nontrivial Hamiltonian knot with a positive value of $a_{2}$ (Hashimoto-Nikkuni \cite[Corollary 1.10]{HN14}). Since there are two hundred and eighty subgraphs of $K_{8}$ isomorphic to $K_{3,3,1,1}$ and for any $8$-cycle $\gamma$ of $K_{8}$ there exist thirty six subgraphs of $K_{8}$ isomorphic to $K_{3,3,1,1}$ containing $\gamma$, we have that there are at least $\lceil 280/36\rceil = 8$ nontrivial Hamiltonian knots with a positive value of $a_{2}$ in every rectilinear spatial graph of $K_{8}$. 
\end{Remark}

\begin{Problem}\label{prob2}
Determine the minimum number of nontrivial Hamiltonian knots (with a positive value of $a_{2}$) in every rectilinear spatial graph of $K_{n}$ for each $n\ge 8$.  
\end{Problem}

We also refer the reader to \cite{FM09}, \cite{AMT13} and \cite{gordian} for a study of counting nontrivial knots and nonsplittable links in a spatial graph of $K_{n_{1},n_{2},\ldots,n_{k}}$. In particular, a computer program {\it Gordian} \cite{gordian} is very useful, which enables us to calculate the values of $a_{2}$ for all constituent knots and ${\rm lk}$ for all constituent $2$-component links in a spatial complete graph without difficulty.


%
{\normalsize
}

\end{document}